\documentclass[12pt,reqno]{amsart}
\usepackage[utf8]{inputenc}
\usepackage{amsmath,amssymb}
\usepackage{amsthm}
\usepackage{color}
\usepackage{mathtools}
\usepackage{bbm}
\usepackage{graphicx}
\usepackage{color}
\usepackage{amsfonts}
\usepackage[title]{appendix}
\usepackage{tikz}
\usepackage[T1]{fontenc}
\usepackage[utf8]{inputenc}

\usepackage{bm,fullpage}

\numberwithin{equation}{section}

\newtheorem{thm}{Theorem}[section]
\newtheorem{cor}[thm]{Corollary}

\newtheorem{prop}[thm]{Proposition}
\newtheorem{alg}[thm]{Algorithm}
\theoremstyle{definition}
\newtheorem{defn}[thm]{Definition}
\theoremstyle{remark}
\newtheorem{rem}[thm]{Remark}
\numberwithin{equation}{section}

\newcommand{\del}{\delta}

\newcommand{\R}{{\mathbb R}}

\newcommand{\Q}{{\mathbb Q}}
\newcommand{\N}{{\mathbb N}}

\newcommand{\Z}{{\mathbb Z}}

\newcommand{\calD}{{\mathcal D}}

\newcommand{\calF}{{\mathcal F}}
\newcommand{\calG}{{\mathcal G}}

\newcommand{\calL}{{\mathcal L}}

\makeatletter
\newcommand{\doublewidetilde}[1]{{%
  \mathpalette\double@widetilde{#1}%
}}
\newcommand{\double@widetilde}[2]{%
  \sbox\z@{$\m@th#1\widetilde{#2}$}%
  \ht\z@=.9\ht\z@
  \widetilde{\box\z@}%
}
\makeatother

\title{On the translates of general dyadic systems on $\R$}%

\author{Theresa C. Anderson, Bingyang Hu, Liwei Jiang, Connor Olson and Zeyu Wei}

\address{Theresa C. Anderson: Department of Mathematics, Purdue University, 150 N. University St., W. Lafayette, IN 47907, U.S.A.}%
\email{tcanderson@purdue.edu}

\address{Bingyang Hu: Department of Mathematics, University of Wisconsin, Madison, 480 Lincoln Dr., Madision, WI 53705, U.S.A.}%
\email{bhu32@wisc.edu}
\address{Liwei Jiang: School of Operations Research and Information Engineering, Cornell University, 206 Rhodes Hall, Ithaca, NY 14850, U.S.A.}%
\email{lj282@cornell.edu}
\address{Connor Olson: Department of Mathematics, Pennsylvania State University, University Park, McAllister Building, Pollock Road, University Park, PA 16802, U.S.A.}%
\email{cjo5325@psu.edu}
\address{Zeyu Wei: Department of Statistics, University of Washington, Box 354322, 
Seattle, WA, 98195-4322, U.S.A.}%
\email{zwei5@uw.edu}

\date{\today}%

\thanks{This project was partially supported by NSF DMS 1502464, NSF DMS 1600458, an NSF RTG in Analysis and Applications, NSF grant 1500182, and an NSF Graduate Research Fellowship (for T. C. Anderson).}

\begin{document}
\maketitle


\begin{abstract}
Many techniques in harmonic analysis use the fact that a continuous object can be written as a sum (or an intersection) of dyadic counterparts, as long as those counterparts belong to an adjacent dyadic system. Here we generalize the notion of adjacent dyadic system and explore when it occurs, leading to some new and perhaps surprising classifications.  In particular, we show that every dyadic grid is determined by two parameters, the \emph{shift} and the \emph{location}; moreover two dyadic grids form an adjacent dyadic system if and only if their shifts and locations satisfy readily verifiable conditions. 
\end{abstract}

\section{Introduction}

\bigskip

Dyadic techniques play an important role in harmonic analysis. The key idea for many of these techniques is to allow one to understand some certain object (for example, operator, function space, etc) via its dyadic version, which is often much more fruitful and easier to handle (see \cite{LPW}, \cite{P}, \cite{PW} to name just a few). 

One of the recent successful applications of such an idea is in the proof of $A_2$ conjecture by using sparse domination (see, e.g., \cite{AL}), which states that for $T$ an $L^2$ bounded Calder\'on-Zygumund operator and $w \in A_2$, it holds that 
$$
\|T\|_{L^2(w)} \le C(n, T) [w]_{A_2}. 
$$
A key fact used to prove the above result is called Mei's lemma, which says there are $2^d$ dyadic grids $\calD_\alpha$ such that for any cube $Q \subset \R^d$, there exists a cube $Q_\alpha$ in one of the grids such that  $Q \subset Q_\alpha$ and $\ell_{Q_\alpha} \le 6 \ell_Q$, where all these $2^d$ dyadic grids can be regarded as a translate of the standard dyadic grid. 

The topic of Mei's lemma and the related concept of adjacent dyadic systems, which is the main subject of this paper, has been often used in the literature (and often misquoted).  One of the most detailed in correcting the history is the recent monograph of Cruz-Uribe \cite{DCU}, which also explains how these concepts fit into dyadic harmonic analysis as a whole.  See also the monograph of Lerner and Nazarov \cite{LN}.  Extensions of Mei's lemma to spaces of homogeneous type, including some background about the attributions given it by others, appear in Hyt\"onen and Kairema \cite{HK}.  Their work, builiding on the referenes therein, was important to extensions of the $A_2$ theorem mentioned above (see, e.g., \cite{AV}) by describing in great detail how to construct a non-random dyadic system on a space of homogeneous type.   This paper addresses the extension of Mei's lemma in a different sense; we ask, and give a precise answer to the question "what are all the translates of a general dyadic system on $\mathbb{R}$ such that Mei's lemma holds?"

In this paper, we give a complete answer to this question on the real line as stated specifically in the final two theorems, Theorem \ref{thm03} and Theorem \ref{mainthm}.  Informally, these theorems state that two dyadic systems are \emph{adjacent}, which imply that they satisfy a general version of Mei's lemma, if and only if the two defining parameters of \emph{shift} and \emph{location} satisfy certain precise conditions that can be quickly checked in practice.

We begin with a few definitions that will be used throughout this paper, in particular to define the aforementioned condition on the shift parameter of a grid.

\begin{defn} \label{defn00}
Given $n \in \N, n \ge 2$, a collection $\calG$ of left-closed and right-open intervals on $\R$ is called \emph{a general dyadic grid with base $n$ (or $n$-grid)} if the following conditions are satisfied:
\begin{enumerate}
\item [(i).] For any $Q \in \calG$, its sidelength $\ell_Q$ is of the form $n^k, k \in \Z$;
\item [(ii).] $Q \cap R \in \{Q, R, \emptyset\}$ for any $Q, R \in \calG$;
\item [(iii).] For each fixed $k\in \Z$, the intervals of a fixed sidelength $n^k$ form a partition of $\R$.
\end{enumerate}
Moreover, we write $\calG_s$ as the \emph{standard dyadic grid with base $n$}. Namely, 
$$
\calG_s:=\left\{ \left[ \frac{k}{n^m}, \frac{k+1}{n^m} \right) \bigg | k, m \in \Z \right\}. 
$$
\end{defn}

Note that when $n=2$, we get the classic dyadic system on $\R$.

\begin{defn}
\label{Sigma def}
We write $\Sigma$ to be the collection of pairs of integers $(m, k)$ with either $m \ge 0, k \in \Z$ or $m<0, k \neq 0$. 
\end{defn}

This leads us to define our shift parameter $\delta$.
\begin{defn} \label{defn01}
A real number $\del$ is \emph{$n$-far} if the distance from $\del$ to each given rational $k/n^m$ is at least some fixed multiple of $1/n^m$, where $(m, k) \in \Sigma$. That is, if there exists $C>0$ such that
\begin{equation}
\label{C delta}
   \left| \del-\frac{k}{n^m} \right| \ge \frac{C}{n^m}, \quad \forall (m, k) \in \Sigma.
\end{equation}
where $C$ may depend on $\del$ but independent of $m$ and $k$.  We will refer to the best constant in \eqref{C delta} as $C(\delta)$.
Finally, we denote $\calF_n$ to be the collection of $n$-far numbers. 
\end{defn}

\begin{rem}\label{rem1}
We have the following observations:

\begin{enumerate}  
\item [1.] The definition of the $2$-far number in \cite[Definition 1.2]{LPW} is not quite accurate: one needs to exclude the case when $k=0$ when $m$ is very negative, otherwise the set of far numbers is empty.  To illustrate this, let $n = 2$, $k = 0$, $m<0$ and fix $\del$.  The equation \eqref{C delta} gives that $|\del| \geq \frac{C}{2^m}$, and as $m \to -\infty$, this forces $C \to 0$.  This is also the reason for us to consider the set $\Sigma$ in our definition of far numbers, rather than all pairs of integers;
\item [2.] 
An easy computation shows that  we can replace the condition \eqref{C delta} in Definition \ref{defn01} by the equivalent
$$
d(\del):= \inf \left\{n^m \left|\del-\frac{k}{n^m} \right| : m \ge 0, k \in \Z \right\}>0,
$$
which coincides with the original definition of $2$-far numbers on the circle in \cite{TM}; we will use this constant in computations.  Therefore $\del$ is $n$-far if an only if there exists $C(\del) >0$ which occurs if and only if there exists $d(\del)>0$.  However, for the purpose of dealing with the translates of $\calG_s$ later (which involve certain shifts of the $m$-th generation of $\calG$, where $m<0$),  we consider $C(\delta)$ as a better choice, as it contains some information for negative $m$;
\item [3.] The following assertions are equivalent, by a straightforward calculation:
\begin{enumerate}
\item [a.] $\del$ is $n$-far;
\item [b.] $n^k\del$ is $n$-far, where $k \in \N, k \ge 1$;
\item [c.] $\del$ is $n^q$-far, where $q \in \N, q \ge 1$;
\end{enumerate}
\item [4.] It is also easy to see that if $\del$ is $n$-far, then $\del+1$ is $n$-far, since $d(\del+1)=d(\del)$.  Thus, we may restrict our interest to those $\del \in [0, 1)$ when dealing with $d(\delta)$.  We emphasize that this does not hold for the constant $C(\del)$.  For example, consider the $2$-far number $\del = 1/3$.  We have $d(1/3) = d(10/3) = 1/3$, but $C(1/3) = 1/3$ while $C(10/3)$ is determined by $|10/3 - 2| = 4/3$, so $C(10/3) = 8/3$.  Therefore $C(\del)$ is not translation invariant, whereas $d(\del)$ is.
\item[5.] By the statements in Section 2, we see that the set of far numbers is nonempty and furthermore that if $\del \in [0,1)$ is $n$-far then $C(\del) \leq \frac{1}{n}$.  For any $\del \in \R$ that is $n$-far we will also see that $d(\del) \leq \frac{1}{n}$.
\end{enumerate}
\end{rem}

There are many applications of our results.  For example, our results can be used to expand many of the theorems in \cite{LPW}.  The weight classes $A_p$ and reverse H\"older classes, as well as function classes such as bounded mean oscillation and maximal functions are considered in \cite{LPW}.  We leave the details to the interested reader to pursue.

The structure of the paper is as follows.  Section 2 begins with several examples of far numbers, leading to a complete classification of far numbers, and concludes with some properties.  Section 3 answers the question posed in the introduction by showing precisely which translates satisfy Mei's lemma, hilighted by Theorems \ref{thm03} and \ref{mainthm}.

\bigskip

\section{Far numbers}

\bigskip
We begin this section with a few illustrative examples of far numbers from the first author's Ph.D. thesis \cite{A}.  The proofs use elementary number theory techniques.

\begin{prop}
Let $n$ be prime, and $p$ be a prime such that $n\neq p$.  We have that $1/p$ is n-far.
\end{prop}
\begin{proof}
We will show that 
\[\left|\frac{n^m - kp}{n^m}\right|\geq Cp/n^m\]
which means that \[n^m \geq p(C+k) \text{ or } n^m \leq p(k - C).\]
Now let $C = 1/p$.  Then for all $n, k$, we need $n^m \geq 1+kp$ or $n^m\leq pk-1$.  Since $n^m\in\Z$, the only time our requirement would not be satisfied is when $n^m = kp$.   So for $1/p$ to be n-far, we need \[n^m\not\equiv 0 \pmod{p}\] for all $n$.  However, we have $gcd(p,n) = 1$, so $gcd(n^m,p) = 1$, which implies that $n^m\not\equiv 0 \pmod{p}$ for all $n$.

\end{proof}
\begin{rem}
The choice of $C$ above is actually the best possible, which can be shown using Fermat's Little Theorem (FLT).  Indeed, we show that if $C>1/p$, then the definition of $n$-far fails for $p$ and $n$ (where $\gcd(p,n) = 1$).  So if $C = 1/p+\varepsilon$ for any $\varepsilon >0$, then we would need to show that \[n^m\geq pk+1+p\varepsilon \text{ or } n^m\leq pk - 1-p\varepsilon.\]  Thus we would need to prevent $n^m = pk+1$ and $n^m = pk-1$ for all $n$ and $k$.  But by FLT, \[n^{p-1}\equiv 1\mod{p},\] so \[n^{p-1} = pk+1,\] which is a contradiction.  Thus $C = 1/p$ is the best possible.
\end{rem}

\begin{rem}
In a similar manner, we can also prove that $1/b$ is n-far where $n$ does not divide $b$.  Details are left to the reader.
\end{rem}

\begin{prop}
We have that $\frac{h}{n^j}+\frac{1}{p}\frac{l}{n^j}$ is n-far for $\gcd(p,nl) = 1$.
\end{prop}
\begin{proof}
We must show that \[\left|\frac{h}{n^j}+\frac{1}{p}\frac{l}{n^j}-\frac{k}{n^m}\right|\geq C/n^m\] that is, \[ \left|\frac{hn^{m-j}p+ln^{m-j}-kp}{n^m}\right|\geq Cp/n^m.\]
We must show that there are no integers in the range $(\frac{n^{m-j}(l+hp)}{p}-C, \frac{n^{m-j}(l+hp)}{p}+C)$.  Letting $C = 1/p$, we can easily check that there are no integers in the range $(\frac{n^{m-j}(l+hp)-1}{p}, \frac{n^{m-j}(l+hp)+1}{p})$ since $\gcd(p,n^{m-j}(l+hp)) = 1$ so $(n^{m-j}(l+hp))^{p-1} \equiv 1 \pmod{p}$.
\end{proof}

From now on, we fix some $n \in \N, n \ge 2$. In this section, we characterize all the $n$-far numbers and study the set $\calF_n$. For any $\del \in [0, 1]$, consider its base-$n$ representation, namely,
$$
\del=\sum_{i=1}^\infty \frac{a_i}{n^i}, 
$$
where $ a_i \in \{0, 1, \dots, n-1\}, i \ge 1$ and the choice of $\{a_i\}_{i \in \N}$ is the finest, in the sense that if $\del=\sum\limits_{i=1}^\infty \frac{b_i}{n^i}$ is another base-$n$ representation, then there exists some $i_0 \in \N$, such that $a_i=b_i$ when $i<i_0$ and $a_i>b_i$ when $i=i_0$.  Hence, we can write $\del=(a_1, \dots, a_i, \dots)_n$.

\begin{defn}
Let $\del=(a_1, \dots, a_i, \dots)_n$. If for some $i_2 \ge i_1 \ge 1$, we have
$$
a_{i_1}=a_{i_1+1}= \dots= a_{i_2}=0 \ \textrm{or} \ n-1,
$$
then we say $(a_{i_1}, \dots, a_{i_2})$ is a \emph{tie}. Moreover, we denote
$$
T(\del):=\sup \left\{ i_2-i_1+1 \ \bigg |  \ (a_{i_1}, \dots, a_{i_2}) \ \textrm{is a tie} \right\}. 
$$
That is, $T(\del)$ is the supremum of the \emph{lengths} of all ties. 
\end{defn}
\begin{rem}
We can define ties for all $\del \in \R$ by using the translation invariant property $T(\del) = T(1+\del)$, when can be seen by the definition of the base $n$ expansion.
\end{rem}
\begin{rem}
We give a few examples of ties and provide a brief discussion.  Fix $n = 2$ and note that $T(1/3) = 1$ since $1/3 = (1,0,1,0, \dots )$, while $T(4/7) = 2$ since $4/7 = (1,0,0,1,0,0, \dots)$.  We can show that $4/7$ is far by a variant of the number theoretic arguments given earlier in this section: by choosing $C(4/7) = 1/7$, we have to show that $|\frac{4\cdot 2^m - 7k}{2^m}| \geq \frac{1}{2^m}$, which reduces to showing that $2^b \neq 0 \mod 7$, for $b \in \N$, which always holds.  We note that $T(1/2) = \infty$, and since $1/2 = (1,0,0,0,0, \dots) = (1,0,1,1,1,1, \dots)$, these two representations share the same tie length, which is related to the familiar fact of non-unique representations of decimals.
\end{rem}

We are now ready to present the main result in this section, which completely classifies $n$-far numbers. 
\begin{thm} \label{thm01}
Let $\delta \in [0,1)$.  Then $\del$ is $n$-far if and only if $T(\del)<\infty$. Moreover, if $\del$ is $n$-far, then
\begin{equation} \label{exeq01}
\frac{1}{n^{T(\del)+1}} \le C(\del) \le \frac{1}{n^{T(\del)}}.
\end{equation}
\end{thm}

\begin{rem}
If one instead uses $d(\del)$ instead of $C(\del)$, one gets that, using the translation invariance described in Remark \ref{rem1}, a complete characterization of all $n$-far numbers on $\R$.
\end{rem}

\begin{proof}
\textit{Necessity.} Expecting a contradiction, assume $T(\del)=\infty$, where $\del$ is $n$-far, that is, there exists some $C(\del)>0$, such that
\begin{equation} \label{eq03}
\left| \del-\frac{k}{n^m} \right| \ge \frac{C(\del)}{n^m}, \quad \forall (m, k) \in \Sigma.
\end{equation}
Take and fix some $N$ sufficiently large such that $n^N C(\del)>1$. Since $T(\del)=\infty$, there exists some $\ell \ge 1$ and $M\geq N$, such that
\begin{equation} \label{eq02}
a_{\ell+1}=\dots=a_{\ell+M}=0 \ \textrm{or} \ n-1,
\end{equation}
with $a_\ell \neq a_{\ell+1}$ and $a_{\ell+M} \neq a_{\ell+M+1}$. We consider two different cases. 

\medskip

\textit{Case I: $a_{\ell+1}=\dots=a_{\ell+M}=0$.}

\medskip

From \eqref{eq02}, we have $a_\ell, a_{\ell+M+1} \neq 0$. Consider the base-$n$ representation
$$
(a_1, \dots, a_\ell, 0, \dots, 0, \dots)_n, 
$$
which clearly can be written as $\frac{l_1}{n^\ell}$ for some $l_1 \in \N, l_1<n^\ell$. Then we have 
$$
0 < \del-\frac{l_1}{n^\ell} \le \frac{1}{n^{\ell+M}} \le \frac{1}{n^{\ell+N}}<\frac{C(\del)}{n^\ell},
$$
which is a contradiction. 

\medskip

\textit{Case II: $a_{\ell+1}=\dots=a_{\ell+M}=n-1$.}

\medskip

Again from our early assumption, $a_\ell, a_{\ell+M+1} \neq n-1$, that is, $a_\ell, a_{\ell+M+1} \le n-2$. Consider the base-$n$ representation
$$
(a_1, \dots, a_\ell+1, 0, \dots, 0, \dots)_n,
$$
which is equal to $\frac{l_2}{n^\ell}$ for some $l_2 \in \N, l_2<n^\ell$. Then, 
$$
0 < \frac{l_2}{n^\ell}-\del \le \frac{1}{n^{\ell+M}} \le \frac{1}{n^{\ell+N}} < \frac{C(\del)}{n^\ell},
$$
which, again, contradicts \eqref{eq03}. 

\medskip

\textit{Sufficiency.} Let $T(\del)=M<\infty$. Without the loss of generality, we may assume $0<\del<1$, since $T(0)=\infty$. We have to show that $\del$ is $n$-far, namely, there exists some $C>0$, such that
$$
\left| \del-\frac{k}{n^m} \right| \ge \frac{C}{n^m}, \quad \forall (m, k) \in \Sigma. 
$$
Consider first the case $m<0$ and $k \neq 0$; we have that
$$
\inf_{k \neq 0} \lim_{m \to -\infty} \left( n^m \left|\del-\frac{k}{n^m} \right| \right)=1,
$$
where the infinum is obtained for $k=1$.  This implies that $\del$ satisfies \eqref{C delta} when restricted to these $m$ and $k$, with $C(\delta) = 1$. 
Therefore, it suffices for us to consider the case when $m \ge 0$ and $k \in \N$.
 Furthermore, we can assume that $0 \le \frac{k}{n^m} \le 1$, and hence we have the base-$n$ representation
$$
\frac{k}{n^m}=(b_1, \dots, b_m, 0, \dots, 0, \dots)_n. 
$$
Again, we consider two cases. 

\medskip

\textit{Case I: $\del>\frac{k}{n^m}$.}

\medskip

Since $T(\del)=M$, the length of the ties consisting of $0$ after $a_m$ will not surpass $M$.  If $a_{m+1} \neq b_{m+1}$, then $|\del - \frac{k}{n^m}| > \frac{1}{n^{m+2}}$, and if $a_{m+1} = b_{m+1} = 0$, in the worst case, $a_{m+1} = \dots = a_{m+M} = 0$, which implies
$$
\left| \del-\frac{k}{n^m}\right|>\frac{1}{n^{m+M+1}}=\frac{1}{n^{M+1}} \cdot \frac{1}{n^m}.
$$

\medskip

\textit{Case II: $\del<\frac{k}{n^m}$.}

\medskip

Again, since $T(\del)=M$, the length of the ties consisting of $n-1$ after $a_m$ will not surpass $M$, which implies
$$
\left| \del-\frac{k}{n^m}\right|=\frac{k}{n^m}-\del>\frac{1}{n^{m+M+1}}=\frac{1}{n^{M+1}} \cdot \frac{1}{n^m}.
$$
The proof is complete if we put $C=\frac{1}{n^{M+1}}$. 

\medskip

Finally, for the estimate \eqref{exeq01}, it is easy to see that the first inequality follows from the sufficient part, while the second one follows from the proof of the necessary part.

\end{proof}

An interesting consequence of Theorem \ref{thm01} is that there exists some irrational far numbers. More precisely, we have the following corollary. 

\begin{cor}
The following assertions hold.
\begin{enumerate}
\item [(a).] All rationals except those of the form $\frac{k}{n^m}, (m, k) \in \Sigma$ are $n$-far numbers;
\item [(b).] Not all irrationals are $n$-far numbers, meanwhile, the set of irrational $n$-far numbers is not empty.
\end{enumerate}
\end{cor}

\begin{proof}
The assertion (a) is obvious, since the base-$n$ representation behaves periodically after a certain digit. For assertion (b), we give some examples. For an irrational which is not $n$-far, consider the base-$n$ representation
$$
(1, 0, 1, 0, 0, 1, \dots, 1, \underbrace{0, \dots, 0, }_\text{k copies of $0$} 1, \dots)_n,
$$
while for an irrational $n$-far number, consider 
$$
(1, 0, 1, 0, 0, \dots, \underbrace{1, 0, 1, 0, \dots, 1, 0, 0, }_\text{k copies of $1, 0$ and 1 copy of $1, 0, 0$} \dots)_n.
$$
Clearly, these examples work due to Theorem \ref{thm01}.
\end{proof}

We conclude this section by exploring some properties of the set $\calF_n$, which is a generalization of the dyadic case.

\begin{prop}
The set $\calF_n$ is dense and meager in $\R$ with Lebesgue measure zero. 
\end{prop}

\begin{proof}
It suffices for us to consider the set $\calF_n \cap [0, 1)$ (we still denote it as $\calF_n$), since $\del$ is $n$-far if and only if $\del+1$ is $n$-far. 

\medskip

\textit{$\bullet$ $\calF_n$ is dense in $[0, 1)$.}

\medskip

This is clear since the set $\Q \backslash \left\{ \frac{k}{n^m}, (m, k) \in \Sigma \right\}$ is dense in $[0, 1)$. 

\medskip

\textit{$\bullet$ $\calF_n$ has Lebesgue measure zero.}

\medskip

Let $m^*$ and $m_*$ be the outer measure and inner measure induced from the Lebesgue measure on $[0, 1)$, respectively. For each $l \ge 1$, let
$$
A_l:=\left\{ \del \in [0, 1): \left| \del-\frac{k}{n^m} \right| \ge \frac{1}{l} \cdot \frac{1}{n^m}, \forall m \ge 0, k \in \Z \right\}.  
$$
Hence $\calF_n=\bigcup\limits_{l=1}^\infty A_l$. The claim will follow if we can show $m^*(\calF_n)=0$. We prove it by contradiction. Assuming $m^*(\calF_n)>0$, then there exists some $l>0$, such that $m^*(A_l)>0$ and hence we can find a measurable set $B \subset [0, 1)$ with $A_l \subset B$ and $|B|=m^*(A_l)$.

Since $\chi_B$ is measurable, an application of Lebesgue differentiation theorem yields that
$$
\lim_{h \to 0} \frac{1}{2h} \int_{x-h}^{x+h} \chi_B(y)dy=
\begin{cases}
1, & x \in B; \\
0,  & x \notin B,
\end{cases}
\quad a.e.
$$
Since $m^*(A_l)=|B|>0$, we can take and fix some $x_0 \in A_l$, such that
\begin{equation} \label{eq04}
\lim_{h \to 0} \frac{1}{2h} \int_{x_0-h}^{x_0+h} |\chi_B(y)-1| dy=0.
\end{equation}
For each $j \in \N, j \ge 10$, let $I_j$ be an interval of the form
$$
\left[ \frac{m_j-1}{n^j}, \frac{m_j+1}{n^j} \right], \quad \textrm{for some} \ m_j \in \{1, 2, \dots, n^j-1\}, 
$$
which  satisfies $x_0 \in I_j$. Indeed, for each fixed $j $, there are at most two possibilities of $m_j$ and we can pick any of them and then fix our choice (note that since $x_0$ is $n$-far, it cannot take the form $\frac{k}{n^m}$, which implies we have at most two choices). We claim that
$$
\lim_{j \to \infty} \frac{1}{|I_j|} \int_{I_j} \chi_B(y)dy=1. 
$$
Indeed, for each $j \ge 10$, we have
$$
\frac{1}{|I_j|} \int_{I_j} |\chi_B(y)-1|dy\le 2\cdot \frac{1}{4n^{-j}} \int_{x_0-\frac{2}{n^j}}^{x_0+\frac{2}{n^j}} |\chi_B(y)-1| dy,
$$
where by \eqref{eq04},  the right hand side converges to zero as $j \to \infty$. Hence, the claim follows. Thus, we can pick a $j_0$ large enough, such that
\begin{equation} \label{eq05}
\frac{\left|B \cap I_{j_0} \right|}{2n^{-j_0}}>1-\frac{1}{10l}. 
\end{equation}
Fix $j_0$.  Since $|B|=m^*(A_l)+m_*(B \backslash A_l)$ and $|B|=m^*(A_l)$, we have $m_*(B \backslash A_l)=0$ and in particular, 
$$
m_*((B \backslash A_l) \cap I_{j_0})=0,
$$
which, combining with the fact that $|B \cap I_{j_0}|=m^*(A_l \cap I_{j_0})+m_*((B \backslash A_l) \cap I_{j_0})$, implies that
$$
|B \cap I_{j_0}|=m^*(A_l \cap I_{j_0}). 
$$
Hence, by \eqref{eq05}, we have
\begin{equation} \label{eq06}
\frac{m^*(A_l \cap I_{j_0})}{2n^{-j_0}}>1-\frac{1}{10l}. 
\end{equation}
However, by the definition of $A_l$, 
$$
\left( \frac{m_0}{n^{j_0}}-\frac{1}{ln^{j_0}},  \frac{m_0}{n^{j_0}}+\frac{1}{ln^{j_0}} \right) \nsubseteq A_l,
$$
which implies
$$
\frac{m^*(A_l \cap I_{j_0})}{2n^{-j_0}} \le \frac{\frac{2}{n^{j_0}}-\frac{2}{ln^{j_0}}}{\frac{2}{n^{j_0}}}=1-\frac{1}{l},
$$
which contradicts \eqref{eq06}.

\medskip

\textit{$\bullet$ $\calF_n$ is meager.}

\medskip

Since $\calF_n=\bigcup\limits_{l=1}^\infty A_l$, it suffices to show $A_l$ is a nowhere dense set for each $l$. Assume it does not hold for some $l \ge 0$. Then the closure of $A_l$ contains some open set $I$, in particular, it contains an interval of the form $\left(\frac{k-1}{n^m}, \frac{k+1}{n^m} \right)$ for some $m$ sufficiently large and $k \in  \{1, \dots n^m-1\}$. However, by the definition of $A_l$, it does not contain the interval $\left( \frac{k}{n^m}-\frac{1}{ln^m}, \frac{k}{n^m}+\frac{1}{ln^m} \right)$, 
which is a contradiction.
\end{proof}

\bigskip

\section{Translates of the generalized dyadic grids and Mei's lemma}

\bigskip

In this section, we study the translates of the general dyadic grids and generalize Mei's lemma. 

\bigskip

\subsection{Translates of $\calG_s$}

\bigskip

Recall $\calG_s$ is the standard dyadic grid with base $n$, which is defined as
$$
\calG_s=\left\{ \left[ \frac{k}{n^m}, \frac{k+1}{n^m} \right)  \bigg | k, m \in \Z \right\}. 
$$

\begin{defn} \label{defn03}
For any $\del \in \R$, the \emph{translated grid $\calG_s^\del$} of $\calG_s$ is defined as follows:
\begin{enumerate}
\item [(1). ] For $m \ge 0$, the $m$-th generation of $\calG_s^\del$ is defined as 
$$
\calG_{s, m}^\del:=\left\{ \left[ \del+\frac{k}{n^m}, \del+\frac{k+1}{n^m} \right) \bigg | k \in \Z \right\};
$$
\item [(2). ] For $m<0$, $m$ even, the $m$-th generation of  $\calG_s^\del$ is defined as 
$$
\calG_{s, m}^\del:=\left\{ \left[ \del+\sum_{j=\frac{m}{2}+1}^0 \frac{1}{n^{2j}}+\frac{k}{n^m},  \del+\sum_{j=\frac{m}{2}+1}^0 \frac{1}{n^{2j}}+\frac{k+1}{n^m} \right) \bigg | k \in \Z \right\};
$$
\item [(3). ] For $m<0$, $m$ odd, the $m$-th generation of $\calG_s^\del$ is defined as
$$
\calG_{s, m}^\del:=\left\{ \left[ \del+\sum_{j=\frac{m-1}{2}+1}^0 \frac{1}{n^{2j}}+\frac{k}{n^m}, \del+\sum_{j=\frac{m-1}{2}+1}^0 \frac{1}{n^{2j}}+\frac{k+1}{n^m} \right)  \bigg | k \in \Z \right\}.
$$
\end{enumerate}
\end{defn}
It is easy to see that $\calG_s^\del$ is a general dyadic grid with base $n$, as in Definition \ref{defn00}. Also note that in general, $\calG_s^0$ and $\calG_s$ are not the same grid.

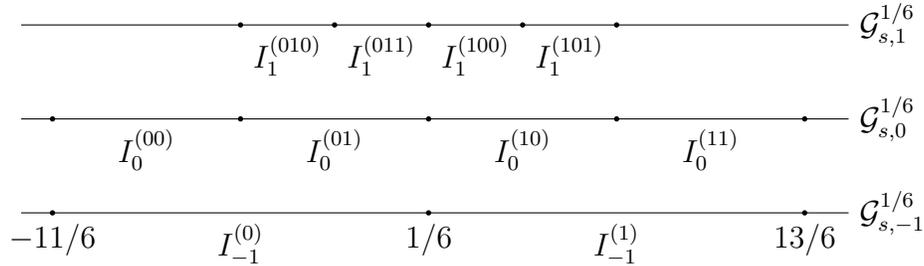
\begin{figure}[ht]
\begin{tikzpicture}[scale=2.5]
\draw(-2, -.5) --(2.4, -.5) node [right] {$\calG_{s, -1}^{1/6}$};
\draw(-2, 0) --(2.4, 0) node [right] {$\calG_{s, 0}^{1/6}$};
\draw(-2, .5) --(2.4, .5) node [right] {$\calG_{s, 1}^{1/6}$};
\fill (1/6, 0) circle [radius=.4pt];
\fill (7/6, 0) circle [radius=.4pt];
\fill (-5/6, 0) circle [radius=.4pt];
\fill (1/6, -.5) circle [radius=.4pt];
\fill (1/6, -.5) node [below] {$1/6$};
\fill (13/6, -.5) node [below] {$13/6$};
\fill (-11/6, -.5) node [below] {$-11/6$};
\fill (13/6, -.5) circle [radius=.4pt];
\fill (-11/6, -.5) circle [radius=.4pt];
\fill (13/6, 0) circle [radius=.4pt];
\fill (-11/6, 0) circle [radius=.4pt];
\fill (-4/3, 0) node [below] {$I^{(00)}_{0}$};
\fill (5/3, 0) node [below] {$I^{(11)}_{0}$};
\fill (1/6, .5) circle [radius=.4pt];
\fill (7/6, .5) circle [radius=.4pt];
\fill (-5/6, .5) circle [radius=.4pt];
\fill (2/3, .5) circle [radius=.4pt];
\fill (-1/3, .5) circle [radius=.4pt];
\fill (-5/6, -.5) node [below] {$I^{(0)}_{-1}$};
\fill (7/6, -.5) node [below] {$I^{(1)}_{-1}$};
\fill (2/3, 0) node [below] {$I^{(10)}_{0}$};
\fill (-1/3, 0) node [below] {$I^{(01)}_{0}$};
\fill (5/12, .5) node [below] {$I^{(100)}_{1}$};
\fill (11/12, .5) node [below] {$I^{(101)}_{1}$};
\fill (-1/12, .5) node [below] {$I^{(011)}_{1}$};
\fill (-7/12, .5) node [below] {$I^{(010)}_{1}$};
\end{tikzpicture}
\caption{$\calG_s^{1/6}$ with base $n=2$.}
\label{sparserangefigure}
\end{figure}

\begin{defn} \label{defn02}
Let $\calG_1$ and $\calG_2$ be two general dyadic grids with base $n$. We say $\calG_1$ and $\calG_2$ are \emph{adjacent} if for each interval $Q \subset \R$, there is an interval $I \subset \R$, such that
\begin{enumerate}
\item[(i). ] $Q \subset I$;
\item [(ii). ] There exists some absolute constant $C>0$, such that $|I| \le C|Q|$;
\item [(iii). ] $I \in \calG_1$ or $I \in \calG_2$.
\end{enumerate}
Namely, the grids $\calG_1$ and $\calG_2$ satisfy a generalized version of Mei's lemma. 
\end{defn}

\begin{rem}
For $n= 2$, the standard dyadic grid, and the $1/3$ shifted dyadic grid, relating back to the ``1/3-trick" stated in the introduction and used in many contexts (such as \cite{AL}) are adjacent.
\end{rem}

\begin{thm} \label{thm02}
The grids $\calG_s$ and $\calG_s^\del$ are adjacent if and only if $\del$ is $n$-far.
\end{thm}

\begin{proof}
\textit{Necessity.} Assume $\del$ is not $n$-far. Since $\calG_s$ and $\calG_s^\del$ are adjacent, there exists some $N \ge 100$, such that for any interval $Q \subset \R$, there exists an interval $I \subset \R$, such that
\begin{enumerate}
\item[(i). ] $Q \subset I$;
\item [(ii). ] $|I| \le n^N|Q|$;
\item [(iii). ] $I \in \calG_s$ or $I \in \calG_s^\del$.
\end{enumerate}
Since $\del$ is not $n$-far, there exists $m_0 \ge 0$ and $k_0 \in \Z$ (see, Remark \ref{rem1}, 2), such that
\begin{equation} \label{eq07}
\left| \del-\frac{k_0}{n^{m_0}} \right|<\frac{1}{n^{N+1}} \cdot \frac{1}{n^{m_0}}.
\end{equation}

Choose an interval $Q$ containing both $\del$ and $\frac{k_0}{n^{m_0}}$ with $|Q|<\frac{1}{n^{N+1}} \cdot \frac{1}{n^{m_0}}$. We claim that there does not exists an interval $I \in \calG_s$ or $I \in \calG_s^\del$, such that the above (i) and (ii) hold, and hence the necessary part is proved.

\medskip

\textit{Proof of the claim.}

\medskip

Assume $I \in \calG_s^\del$. Then $I \in \calG_{s, m}^\del$, for some $m<0$. This is because that $\del$ is an endpoint of the interval
$$
\left[ \del, \del+\frac{1}{n^m} \right) \in \calG_{s, m}^\del, \ \textrm{for} \ m \ge 0. 
$$
However, 
$$
|I| \ge \frac{1}{n^m} \ge n^N \cdot \frac{1}{n^{N+1}} \cdot \frac{1}{n^{m_0}}>n^N |Q|,  
$$
which contradicts (ii) for $m_0>m$.

\medskip

On the other hand, if $I \in \calG_s$, then $I \in \calG_{s, m}$, where $m \le m_0-1$. Otherwise, the point $\frac{k_0}{n^{m_0}}$ will again become an endpoint of some interval contained in $\calG_{s, m}$. Thus,
$$
|I| \ge \frac{1}{n^{m_0-1}}> n^N \cdot \frac{1}{n^{N+1}} \cdot \frac{1}{n^{m_0}}>n^N |Q|,
$$
which contradicts (ii) again.

\medskip

\textit{Sufficiency.} Let $\del$ be a $n$-far number, and without the loss of generality, we may assume that $\del>0$. Then for any interval $Q \subset R$, we need to show all the conditions in Definition \ref{defn02} are satisfied for the grids $\calG_s$ and $\calG_s^\del$. 

We make the following observation first: for each $m<0$, write
$$
A(m):= \sum_{j=\frac{m}{2}+1}^0 \frac{1}{n^{2j}}=\frac{n^{-m}-1}{n^2-1}, \quad m \ \textrm{even},
$$
and
$$
B(m):= \sum_{j=\frac{m-1}{2}+1}^0 \frac{1}{n^{2j}}=\frac{n^{1-m}-1}{n^2-1},  \quad m \ \textrm{odd}. 
$$
Hence, we can find $C_1=C_1(n)$ and $C_2=C_2(n)$ with $0<C_1, C_2<1$, such that
$$
C_1 n^{-m} \le A(m), B(m) \le C_2 n^{-m}. 
$$
Moreover, take and fix some $N \ge 0$, such that for those $m \le -N$ one has 
\begin{equation} \label{eq08}
n^{-m}>\del \quad  \textrm{and} \quad  \frac{\del}{n^{-m}} \le \min\left\{ \frac{C_1}{100}, \frac{1-C_2}{2} \right\}.
\end{equation}

Now consider the sets of endpoints of the intervals in $\calG_s$ and $\calG_s^\del$.  Namely, for any $m \in \Z$, we consider the sets
$$
A_{s, m}:=\left\{ \frac{k}{n^m}, \bigg | k \in \Z \right\}, \ m \in \Z
$$
and
$$
A_{s, m}^\del:=
\begin{cases}
\left\{ \del+\frac{k}{n^m} \bigg | k \in \Z \right\}, & m \ge 0; \\
\\
\left\{ \del+ \sum\limits_{j=\frac{m}{2}+1}^0 \frac{1}{n^{2j}}+\frac{k}{n^m} \bigg | k \in \Z \right\},  & m<0, \ m \ \textrm{even}; \\
\\
\left\{ \del+ \sum\limits_{j=\frac{m-1}{2}+1}^0 \frac{1}{n^{2j}}+\frac{k}{n^m} \bigg | k \in \Z \right\}, & m<0, \ m \ \textrm{odd}. 
\end{cases}
$$

\medskip

\textit{Claim: There exists a constant $C=C(\del)>0$, such that for any $m \in \Z$, 
\begin{equation} \label{eq10}
|a-b| \ge \frac{C(\del)}{n^m},  \quad \textrm{where} \ a, b \in A_{s, m}^\del \cup A_{s, m}, a \neq b. 
\end{equation}
}

\medskip

\textit{Proof of the claim.} We prove the claim by considering different cases.

\medskip

\textit{Case I: $a, b \in A_{s, m}$ or $a, b \in A_{s, m}^\del$.} Clearly, we have
$$
|a-b| \ge \frac{1}{n^m}. 
$$

\medskip

\textit{Case II: $a \in A_{s, m}^\del$ and $b \in A_{s, m}$, $m \ge 0$.} Since $\del$ is $n$-far, we have, for some $k \in \Z$, 
$$
|a-b|= \left| \del-\frac{k}{n^m} \right| \ge \frac{d(\del)}{n^m}.
$$

\medskip

\textit{Case III: $a \in A_{s, m}^\del$ and $b \in A_{s, m}$, $m \le -N$ and $m$ even.} By the definition of $A_{s, m}^\del$ and $A_{s, m}$, we know 
$$
a-b=\del+A(m)+\frac{k}{n^m}
$$
for some $k \in \Z$, which, combining with \eqref{eq08}, implies there exists some $C_1'=C_1'(C_1, C_2)$, such that
\begin{equation} \label{eq09}
|a-b| \ge \frac{C_1'}{n^m}.
\end{equation}

\medskip

\textit{Case IV: $a \in A_{s, m}^\del$ and $b \in A_{s, m}$, $m \le -N$ and $m$ odd.} This case is similar to Case III, with replacing $A(m)$ by $B(m)$, and hence we can still conclude the inequality \eqref{eq08} holds for this case. 

\medskip

\textit{Case V: $a \in A_{s, m}^\del$ and $b \in A_{s, m}$, $-N<m<0$.} Note that we have only finitely many choices for $m$. Using the fact that $\del$ is not an integer, we can find some $C_1''>0$, such that
$$
|a-b| \ge \frac{C_1''}{n^m}.
$$
Thus, the claim follows if we let $C(\del):=\min\left\{1, d(\del), C_1', C_1''\right\}$. 

For any interval $Q \subset \R$, there exists a $m_0 \in \N$, such that
$$
\frac{C(\del)}{n^{m_0+1}} \le |Q|<\frac{C(\del)}{n^{m_0}}.
$$
Then applying \eqref{eq10}, we have for any $a, b \in A_{s,m_0} \cup A_{s, m_0}^\del$, we have
$$
|a-b| \ge \frac{C(\del)}{n^{m_0}}>|Q|.
$$
Therefore, there is at most one point that belongs to both $Q$ and $A_{s, m_0} \cup A_{s, m_0}^\del$, that is, $Q \cap A_{s, m_0}=\emptyset$ or $Q \cap A_{s, m}^\del=\emptyset$. Thus, $Q$ has to be contained in some $I \in \calG_{s, m_0}$ or $I \in \calG_{s, m_0}^\del$, with
$$
|I|=n^{-m_0} \le \frac{n}{C(\del)} \cdot |Q|. 
$$
The proof is complete. 
\end{proof}

\begin{rem}
 We note that \cite[Proposition 2.2]{LPW} is part of a particular case of Theorem \ref{thm02} when $n=2$.  In the proof of \cite[Proposition 2.2]{LPW}, the authors simply chose $C(\del)$ as $d(\del)$, that is, they only consider the Case II above.  However, we also need to consider the contribution of $A(m)$ or $B(m)$ when $m$ is negative: for example, one can choose some $\del>0$ such that $d(\del)>C_1'$, which leads the estimate
$$
|a-b| \ge \frac{d(\del)}{n^m}
$$
failing for Case III and Case IV. 
\end{rem}

\medskip

\subsection{General translates and generalized Mei's lemma}

\medskip

In this section, we generalize Mei's lemma to any pair of general dyadic grids with base $n$. The key observation is that any translate is uniquely determined by a number $\del$ and a location function $\calL$. 

We shall give the definition of $\calL$ first. Consider an infinite sequence
$$
\textbf{a}:=\{a_0, a_1, \dots, a_j, \dots \}
$$
where $a_i \in \{1, 2, \dots, n-1\}$. The \emph{location function associated to $\textbf{a}$,} which maps $\N$ into $\N$, is defined as
$$
\calL_{\textbf{a}}(j):=\sum_{i=0}^{j-1} a_i n^i, \quad j \ge 1
$$
and $\calL_{\textbf{a}}(0)=0$. Clearly, for any $\textbf{a}$, $\calL_{\textbf{a}}(j) \in \left\{0, \dots,  n^j-1 \right\}, \forall j \in \N$ and moreover, we shall see later that $\calL_{\textbf{a}}(j)$ indeed reflects the location of the origin after translating $j$ times. Note that in Definition \ref{defn03}, we are considering the special choice 
$$
\textbf{a}=\{1, 0, 1, 0, \dots, 1, 0, \dots\}
$$
with the location function 
$$
\calL_{\textbf{a}}(j)=\sum_{i=0}^{\left \lfloor{\frac{j}{2}}\right \rfloor} n^{2i},
$$
where $\lfloor \cdot \rfloor$ is the floor function. This observation suggests us to introduce the following definition.

\begin{defn} \label{defn04}
Let $\del \in \R$, $\textbf{a}$ and $\calL_{\textbf{a}}$ defined as above. Let $\calG(\del, \calL_{\textbf{a}})$ be the collection of the following intervals:
\begin{enumerate}
\item [1.] For $m \ge 0$, the $m$-th generation of $\calG(\del, \calL_{\textbf{a}})$  is defined as
$$
\calG(\del, \calL_{\textbf{a}}) _m:=\left\{ \left[ \del+\frac{k}{n^m}, \del+\frac{k+1}{n^m} \right) \bigg | k \in \Z \right\};
$$
\item [2.] For $m<0$, the $m$-th generation of $\calG(\del, \calL_{\textbf{a}})$ is defined as
$$
\calG(\del, \calL_{\textbf{a}})_m:= \left\{ \left[ \del+\calL_{\textbf{a}}(-m)+\frac{k}{n^m}, \del+\calL_{\textbf{a}}(-m)+\frac{k+1}{n^m} \right) \bigg | k \in \Z \right\}. 
$$
\end{enumerate}
Or equivalently, one can also use the set of endpoints of each generation to define $\calG(\del, \calL_{\textbf{a}})$. Namely, 
\begin{enumerate}
\item [1.] For $m \ge 0$, the set of the endpoints of the $m$-th generation is defined as
$$
A(\del, \calL_{\textbf{a}})_m:=\left\{ \del+\frac{k}{n^m} \bigg | k \in \Z \right\};
$$
\item [2.] For $m<0$, the set of the endpoints of the $m$-th generation is defined as
$$
A(\del, \calL_{\textbf{a}})_m:=\left\{ \del+\calL_{\textbf{a}}(-m)+\frac{k}{n^m} \bigg | k \in \Z \right\}. 
$$
\end{enumerate}
\end{defn}
For example, we can write the standard grid $\calG_s$ as  $\calG(0, \calL_{\{0, \dots, 0, \dots \}})$ and $\calG_s^\del$ as $\calG(\del, \calL_{ \{1, 0, \dots, 1, 0, \dots\} })$. 

\begin{prop}\label{prop100}
$\calG(\del, \calL_{\textbf{a}})$ is a general dyadic grid with base $n$. 
\end{prop}

\begin{proof}
It suffices for us to consider the case when $\del=0$, as the grid $\calG(\del, \calL_{\textbf{a}})$ is obtained by traslating the grid $\calG(0, \calL_{\textbf{a}})$ with $\del$ units. 

When $m \ge 0$, it is clear that all the intervals with side length $\frac{1}{n^m}$ are uniquely determined, as they are the dyadic children of the intervals $[k, k+1), k \in \Z$, moreover, these intervals are of the form
$$
\left[ \frac{k}{n^m}, \frac{k+1}{n^m} \right), \ m \ge 0, k \in \Z. 
$$

When $m<0$, it is more convenient to understand the construction in Definition \ref{defn04} inductively. Indeed, when $m=-1$, we see that
$$
\calG(0, \calL_{\textbf{a}})_{-1}=\left\{ \left[a_0+kn, a_0+(k+1)n \right) \bigg | k \in \Z \right\}
$$
for some $a_0 \in \{0, 1, \dots, n-1\}$. Note that the $n$ different choices of $a_0$ have a one-to-one correspondence to the $n$ different ways to choose the dyadic parent of the interval $[0, 1)$ in $(-1)$-th generation, which, once this is fixed, will determine the whole $(-1)$-th generation. Moreover, we can view $\calL_{\textbf{a}}(1)=a_0$ as the location of the origin point after we translate once. Applying this process inductively, when $m \le -2$, we can view
\begin{eqnarray*}
&&\calG(\del, \calL_{\textbf{a}})_m=\bigg\{ \bigg[ \calL_{\textbf{a}}(-m-1)+a_{(-m-1)}n^{-m-1}+kn^{-m}, \\
&& \quad \quad \quad \quad \quad \quad \quad\quad \quad \calL_{\textbf{a}}(-m-1)+a_{(-m-1)}n^{m-1}+(k+1)n^{-m} \bigg) \bigg| k \in \Z \bigg\}. 
\end{eqnarray*}
Again, the $n$ different choices of $a_{(-m-1)}$ corresponds to the $n$ different ways to choose the dyadic parent of the interval 
$$
\left[ \calL_{\textbf{a}}(-m-1), \calL_{\textbf{a}}(-m-1)+n^{-m-1}\right),
$$ 
where the quantity $\calL_{\textbf{a}}(-m)$ represents the location of the original point after we translate $m$ times. From these constructions, it is easy to check that all the conditions in Definition \ref{defn00} are satisfied, so $\calG(0, \calL_{\textbf{a}})$ is a general dyadic grid with base $n$, as is $\calG(\del, \calL_{\textbf{a}})$.
\end{proof}

A natural question to ask is "for what $\del_1, \del_2, \textbf{a}$ and $\textbf{b}$, are the $n$-grids $\calG(\del_1, \calL_{\textbf{a}})$ and $\calG(\del_2, \calL_{\textbf{b}})$ are adjacent?"  Motivated by Theorem \ref{thm02}, we have the following result.

\begin{thm} \label{thm03}
Let $\del_1, \del_2 \in \R$ and $\calL_{\textbf{a}}, \calL_{\textbf{b}}$ be defined as in Definition \ref{defn04}. Then $\calG(\del_1, \calL_{\textbf{a}})$ and $\calG(\del_2, \calL_{\textbf{b}})$ are adjacent if and only if
\begin{enumerate}
\item[(i).] $\del_1-\del_2$ is $n$-far;
\item [(ii).] 
\begin{equation} \label{eq21}
0<\liminf_{j \to +\infty} \left| \frac{\calL_{\textbf{a}}(j)-\calL_{\textbf{b}}(j)}{n^j} \right| \le \limsup_{j \to +\infty} \left| \frac{\calL_{\textbf{a}}(j)-\calL_{\textbf{b}}(j)}{n^j} \right|<1.
\end{equation}
\end{enumerate}
\end{thm}

\begin{rem}
We remark that the standard dyadic grid $\calG$ and the $1/3$ shifted dyadic grid $\calG_s^{1/3} = \calG(1/3, \calL_{ \{1, 0, \dots, 1, 0, \dots\} })$ are adjacent (this is checked in detail preceding Theorem \ref{mainthm}).  The standard dyadic grid $\calG$ and the $1/4$ shifted dyadic grid $\calG_s^{1/4} = \calG(1/4, \calL_{ \{1, 0, \dots, 1, 0, \dots\} })$ are not adjacent since $\frac{1}{2} - \frac{1}{4} = \frac{1}{4}$, which is not $2$-far. 
\end{rem}

\begin{proof}
The proof of the sufficient part and the first half of necessary part is an easy modification of the sufficiency and the necessity of  Theorem \ref{thm02}, respectively, and hence we omit the proof here. Thus, it suffices for us to show that if $\calG(\del_1, \calL_{\textbf{a}})$ and $\calG(\del_2, \calL_{\textbf{b}})$ are adjacent, then the second condition holds.

Since $\calG(\del_1, \calL_{\textbf{a}})$ and $\calG(\del_2, \calL_{\textbf{b}})$ are adjacent, there exists some $C>0$, such that for any interval $Q \subset \R$, there exists an interval $I \subset \R$, such that
\begin{enumerate}
\item [(a).] $Q \subset I$;
\item [(b).] $|I| \le C|Q|$;
\item [(c).] $I \in \calG(\del_1, \calL_{\textbf{a}})$ or $I \in \calG(\del_2, \calL_{\textbf{b}})$.  
\end{enumerate}
We prove the desired result by contradiction. Assume \eqref{eq21} fails and consider two different cases. 

\medskip

\textit{Case I: $\liminf\limits_{j \to +\infty} \left| \frac{\calL_{\textbf{a}}(j)-\calL_{\textbf{b}}(j)}{n^j} \right|=0$.} Take and fix some $j_0>0$, such that
$$
\left| \frac{\del_1-\del_2}{n^{j_0}}\right|<\frac{1}{4C} \quad \textrm{and} \quad  \left| \frac{\calL_{\textbf{a}}(j_0)-\calL_{\textbf{b}}(j_0)}{n^{j_0}} \right|<\frac{1}{4C}. 
$$
Consider the points
$$
p_1:=\del_1+\calL_{\textbf{a}}(j_0), \quad p_2:=\del_2+\calL_{\textbf{b}}(j_0) \in \R.
$$
Clearly, for any $m \ge -j_0$, 
$$
p_1 \in A(\del_1, \calL_{\textbf{a}})_m \ \textrm{and} \ p_2 \in A(\del_2, \calL_{\textbf{b}})_m.
$$
Moreover, $\left|p_1-p_2\right|<\frac{1}{2C} \cdot n^{j_0}$. Then we can choose an open interval $Q$ containing $p_1$ and $p_2$, with length $\frac{n^{j_0}}{2C}$. 

Since $\calG(\del_1, \calL_{\textbf{a}})$ and $\calG(\del_2, \calL_{\textbf{b}})$ are adjacent, we can find some $I \in \calG(\del_1, \calL_{\textbf{a}})$ or $\calG(\del_2, \calL_{\textbf{b}})$, such that the conditions (a), (b) and (c) above are satisfied. Moreover, we have $|I| \ge n^{j_0+1}$. Otherwise, $I \in \calG(\del_1, \calL_{\textbf{a}})_m$ (or $I \in \calG(\del_2, \calL_{\textbf{b}})_m$, respectively) for some $m \ge -j_0$, which is impossible since $p_1 \in A(\del_1, \calL_{\textbf{a}})_m$ (or $p_2 \in A(\del_2, \calL_{\textbf{b}})_m$, respectively). However, 
$$
|I| \ge n^{j_0+1}>\frac{n^{j_0}}{2}=C \cdot \frac{n^{j_0}}{2C}=C|Q|,
$$
which is a contradiction. 

\medskip

\textit{Case II: $\limsup\limits_{j \to +\infty} \left| \frac{\calL_{\textbf{a}}(j)-\calL_{\textbf{b}}(j)}{n^j} \right|=1$.} We may assume that there exists some $j_1> 0$, such that
$$
\left| \frac{\calL_{\textbf{a}}(j_1)-\calL_{\textbf{b}}(j_1)-n^{j_1}}{n^{j_1}} \right|<\frac{1}{4C}. 
$$
Then we consider the points
$$
q_1:=\del_1+\calL_{\textbf{a}}(j_1), \quad q_2:=\del_2+\calL_{\textbf{b}}(j_1)+n^{j_1} \in \R.
$$
The rest of the proof is the same as Case I, and hence we omit it here. 
\end{proof}

As a consequence of Theorem \ref{thm03}, we are able to answer the following general question: given any two $n$-grids $\calG_1$ and $\calG_2$, is there an efficient way to verify $\calG_1 $ and $\calG_2$ are adjacent or not?

We need some preparation. 

\begin{prop} \label{prop200}
Let $\calG$ be any $n$-grid, then there exists some $\del \in \R$ and $\textbf{a}=(a_0, a_1, \dots, a_j, \dots)$, $a_i \in \{0, 1, \dots, n-1\}$, such that $\calG=\calG(\del, \calL_{\textbf{a}})$.
\end{prop}

\begin{proof}
Consider $A_0$, which is the collection of endpoints of the intervals in $\calG_0$. It is clear that there exists only one point belonging to $[0, 1) \cap A_0$, and we fix and label this point $\del$.

Next, from the proof of Proposition \ref{prop100}, we see that $\textbf{a}$ is uniquely determined once we fix the choice of $\del$, as it has a one-to-one correspondence with all the ancestors of $[\del, \del+1)$, which is uniquely determined by $\calG$. 

Hence, $\calG=\calG(\del, \calL_{\textbf{a}})$. 
\end{proof}

\begin{defn}
Let $\del \in \R$ and $\textbf{a}$ be defined as in Proposition \ref{prop200}. We say $\calG(\del, \calL_{\textbf{a}})$ is a \textit{representation of $\calG$} if $\calG=\calG(\del, \calL_{\textbf{a}})$.
\end{defn}

Note that the representation of a $n$-grid may not be unique. For example, one can easily verify that $\calG(0, \calL_{(1, 0, \dots, 0, \dots)})=\calG(2, \calL_{(n-1, n-1, \dots, n-1, \dots)})$.

We are ready to formulate the main result in this section. 

\begin{alg} \label{alg01}
Let $\calG_1$ and $\calG_2$ be any $n$-grids. The following algorithm can be used to check whether $\calG_1$ and $\calG_2$ are adjacent or not.
\begin{enumerate}
\item [\textit{Step I:}] Take any representations of $\calG_1$ and $\calG_2$;
\item [\textit{Step II:}] Check whether these two representations satisfy the conditions in Theorem \ref{thm03}. 
\end{enumerate}
If these conditions are satisfied, then $\calG_1$ and $\calG_2$ are adjacent. Otherwise, they are not. 
\end{alg}

Note that Algorithm \ref{alg01} is well-defined, in the sense that the outcome is independent of the choice of the representations of $\calG_1$ and $\calG_2$. This is guaranteed by Theorem \ref{thm03}. 

\begin{rem}
We illustrate how the above algorithm can be used to check that two grids are adjacent using \ref{thm03}, using our running example of the standard dyadic grid $\calG = \calG(0, \calL_{\{0, \dots, 0, \dots \}})$ and the $1/3$ shifted dyadic grid $\calG_s^{1/3} = \calG(1/3, \calL_{ \{1, 0, \dots, 1, 0, \dots\} })$.  First we check the first condition and see that $\frac{1}{3}$ is $2$-far.  Next we check the second condition and see that
\[
\liminf_{j \to +\infty} \left| \frac{\calL_{\textbf{a}}(j)-\calL_{\textbf{b}}(j)}{n^j} \right|  = \limsup_{j \to +\infty} \left| \frac{\calL_{\textbf{a}}(j)-\calL_{\textbf{b}}(j)}{n^j} \right|  = \sum_{k = 0}^\infty \bigg(\frac{1}{2}\bigg)^{2k+1} = \frac{2}{3}.
\]
Below we will see that if any other representations are chosen for $\calG$ and $\calG_s^{1/3}$, then the $\liminf$ and $\limsup$ above (see \eqref{eq21}) must be either both $\frac{2}{3}$ or both $\frac{1}{3}$.
\end{rem}

Finally, we study a result describing the uniformity of the representation of the grids.  First, let $\calG_1$ and $\calG_2$ be two $n$-grids, which are \emph{not} adjacent, and let $\calG(\del_1, \calL_{\textbf{a}})$ be a representation of $\calG_1$. Then for any representation of $\calG_2=\calG(\del_2, \calL_{\textbf{b}})$ with $\del_1-\del_2$ $n$-far, by Theorem \ref{thm03}, we have either
$$
\liminf_{j \to +\infty} \left| \frac{\calL_{\textbf{a}}(j)-\calL_{\textbf{b}}(j)}{n^j} \right|=0,
$$
or
$$
\limsup_{j \to +\infty} \left| \frac{\calL_{\textbf{a}}(j)-\calL_{\textbf{b}}(j)}{n^j} \right|=1,
$$
which is independent of the choice of a particular representation $\calG_2$.  Hence, we may ask whether there is still some uniformity in the representations of two adjacent grids. It turns out that in a certain sense, there is.  

In Theorem \ref{mainthm} below we show that for adjacent grids the $\liminf$ and $\limsup$ parameters determined by the location functions in \eqref{eq21} must either be identical, or (in a certain sense) inverted, no matter what representation is used.  Particularly, if $C_1$ and $C_2$ denote these (ordered) parameters, any adjacent system will have either matching parameters $C_1$ and $C_2$, or new "inverted" parameters $1-C_2$ and $1-C_1$.  The significance of this is that once the shifts $\delta$ are chosen for the two grids, then the important quantities governed by their location functions in \eqref{eq21} are invariant, up to the permitted inversion.

\begin{thm}
\label{mainthm}
Let $\calG_1$ and $\calG_2$ be two adjacent $n$-grids with representations $\calG(\del_1, \calL_{\textbf{a}})$ and $\calG(\del_2, \calL_{\textbf{b}})$, respectively. Let
$$
C_1:=\liminf_{j \to +\infty} \left| \frac{\calL_{\textbf{a}}(j)-\calL_{\textbf{b}}(j)}{n^j} \right|>0,
$$
and 
$$
C_2:=\limsup_{j \to +\infty} \left| \frac{\calL_{\textbf{a}}(j)-\calL_{\textbf{b}}(j)}{n^j} \right|<1. 
$$
Let further, $\calG(\del_1', \calL_{\textbf{a}'})$ and $\calG(\del_2', \calL_{\textbf{b}'})$ be some other representations of $\calG_1$ and $\calG_2$, respectively. Then, either
$$
\liminf_{j \to +\infty} \left| \frac{\calL_{\textbf{a}'}(j)-\calL_{\textbf{b}'}(j)}{n^j} \right|=C_1 \quad \textrm{and} \quad \limsup_{j \to +\infty} \left| \frac{\calL_{\textbf{a}'}(j)-\calL_{\textbf{b}'}(j)}{n^j} \right|=C_2
$$
or 
$$
\liminf_{j \to +\infty} \left| \frac{\calL_{\textbf{a}'}(j)-\calL_{\textbf{b}'}(j)}{n^j} \right|=1-C_2 \quad \textrm{and} \quad \limsup_{j \to +\infty} \left| \frac{\calL_{\textbf{a}'}(j)-\calL_{\textbf{b}'}(j)}{n^j} \right|=1-C_1.
$$
\end{thm}

\begin{proof}
Since $A(\del_1, \calL_{\textbf{a}})_0=A(\del_1', \calL_{\textbf{a}'})_0$, we have $\del_1'=\del_1+N_1$ for some $N_1 \in \Z$.
Similarly, $\del_2'=\del_2+N_2$ for some $N_2 \in \Z$.

First, we consider the special case when $N_1, N_2 \ge 0$. Since $A(\del_1, \calL_{\textbf{a}})_m=A(\del_1', \calL_{\textbf{a}'})_m$, we have for each $j \in \mathbb{N},  j>0$, 
$$
\del_1+\calL_{\textbf{a}}(j)=\del_1'+\calL_{\textbf{a}'}(j)+d(j)n^j,
$$
where $d(j)$ is some integer depending on $j$.  Since $\del_1'=\del_1+N_1$, it follows that
\begin{equation} \label{eq333}
\calL_{\textbf{a}}(j)=\calL_{\textbf{a}'}(j)+N_1+d(j)n^j, \ j \in \N, \  j>0. 
\end{equation}

\medskip

\textit{Claim: $d(j)=0$ or $-1$ for $j$ large enough.}

\medskip

\textit{Proof of the claim:} Recall that the functions $\calL_{\textbf{a}}$ and $\calL_{\textbf{a}'}$ are non-decreasing and for each $j \ge 0$, the quantities $\calL_{\textbf{a}}(j)$ and $\calL_{\textbf{a}'}(j)$ take the values in $\{0, 1, \dots, n^j-1\}$. Moreover, for each $j \ge 0$, 
$$
-n^j<\calL_{\textbf{a}}(j)-\calL_{\textbf{a}'}(j)<n^j.
$$
These facts suggest that for $l$ large enough, $d(j)$ can only be $-1, 0$ or $1$. We consider the following cases.

\medskip

\textit{Case I: $\lim\limits_{j \to \infty} \calL_{\textbf{a}}(j)<\infty$ and $\lim\limits_{j \to \infty} \calL_{\textbf{a}'}(j)<\infty$. } By \eqref{eq333}, it is clear that $d(j)=0$ for $j$ large enough.

\medskip

\textit{Case II: $\lim\limits_{j \to \infty} \calL_{\textbf{a}}(j)<\infty$ and $\lim\limits_{j \to \infty} \calL_{\textbf{a}'}(j)=\infty$. } Write \eqref{eq333} as 
$$
\calL_{\textbf{a}}(j)-\calL_{\textbf{a}'}(j)=N_1+d(j)n^j,
$$
which implies that $d(j)$ can only equal to $-1$ when $j$ is large. 

\medskip

\textit{Case III: $\lim\limits_{j \to \infty} \calL_{\textbf{a}}(j)=\infty$ and $\lim\limits_{j \to \infty} \calL_{\textbf{a}'}(j)<\infty$. } Again, the equation $ \calL_{\textbf{a}}(j)-\calL_{\textbf{a}'}(j)=N_1+d(j)n^j$ suggests $d(j)$ can only take the value $1$ when $j$ is large. Then for $j$ large enough, 
$$
\calL_{\textbf{a}}(j)-\calL_{\textbf{a}'}(j)<n^j \le N_1+n^j=N_1+d(j)n^j,
$$
which contradicts \eqref{eq333} and Case III cannot happen. 

\medskip

\textit{Case IV: $\lim\limits_{j \to \infty} \calL_{\textbf{a}}(j)=\lim\limits_{j \to \infty} \calL_{\textbf{a}'}(j)=\infty$. }  We claim $0$ is the only possible choice for $d(j)$ for $j$ large enough. If $d(j)=1$ for $j$ large enough, then the same argument in Case III will lead to the desired contradiction; if $d(j)=-1$ for $j$ large enough, then \eqref{eq333} implies
$$
\calL_{\textbf{a}}(j)+n^j=\calL_{\textbf{a}'}(j)+N_1,
$$
which is also impossible since $n^j>\calL_{\textbf{a}'}(j)$ and $\calL_{\textbf{a}}(j)>N_1$ for $j$ large enough. 

Applying the same argument to the grid  $\calG_2$, we have
\begin{equation} \label{eq334}
\calL_{\textbf{b}}(j)=\calL_{\textbf{b}'}(j)+N_2+e(j)n^j, \quad j \in \N, \ j>0,
\end{equation}
where $e(j)$ is some integer depending on $j$ and $e(j)=0$ or $-1$ for $j$ large enough.

Thus, by \eqref{eq333} and \eqref{eq334}, we have
\begin{equation} \label{eq335}
 \frac{\calL_{\textbf{a}'}(j)-\calL_{\textbf{b}'}(j)}{n^j} =  \frac{\calL_{\textbf{a}}(j)-\calL_{\textbf{b}}(j)}{n^j}+\frac{N_2-N_1}{n^j}+\left( e(j)-d(j) \right)
\end{equation}
and hence for large $j$, we have three cases.

\medskip

\textit{Case A: $e(j)-d(j)=0$.} For this case, we have
$$
 \frac{\calL_{\textbf{a}'}(j)-\calL_{\textbf{b}'}(j)}{n^j} =  \frac{\calL_{\textbf{a}}(j)-\calL_{\textbf{b}}(j)}{n^j}+\frac{N_2-N_1}{n^j}, \quad \textrm{for $j$ large enough}.
$$
This implies
$$
\liminf_{j \to +\infty} \left |  \frac{\calL_{\textbf{a}'}(j)-\calL_{\textbf{b}'}(j)}{n^j}  \right|=C_1 \quad \textrm{and} \quad \limsup_{j \to +\infty} \left |  \frac{\calL_{\textbf{a}'}(j)-\calL_{\textbf{b}'}(j)}{n^j}  \right|=C_2.
$$

\medskip

\textit{Case B: $e(j)-d(j)=1$.} By \eqref{eq335}, we have
$$
 \frac{\calL_{\textbf{a}'}(j)-\calL_{\textbf{b}'}(j)}{n^j} =  \frac{\calL_{\textbf{a}}(j)-\calL_{\textbf{b}}(j)}{n^j}+\frac{N_2-N_1}{n^j}+1, \quad \textrm{for $j$ large enough}.
$$
We claim that for this case,
$$
\frac{\calL_{\textbf{a}}(j)-\calL_{\textbf{b}}(j)}{n^j}<0,  \quad \textrm{for $j$ large enough}.
$$
Otherwise, we can find some $j'>0$, such that 
$$
\frac{\calL_{\textbf{a}}(j')-\calL_{\textbf{b}}(j')}{n^{j'}} \ge C_1>0 \quad \textrm{and} \quad  \left| \frac{N_2-N_1}{n^{j'}} \right|<\frac{C_1}{1000}. 
$$
Then we have
$$
 \frac{\calL_{\textbf{a}'}(j')-\calL_{\textbf{b}'}(j')}{n^{j'}}<1< \frac{\calL_{\textbf{a}}(j)-\calL_{\textbf{b}}(j)}{n^j}+\frac{N_2-N_1}{n^j}+1,
$$
which is a contradiction. Thus, for the second case, we have
$$
\liminf_{j \to +\infty} \left| \frac{\calL_{\textbf{a}'}(j)-\calL_{\textbf{b}'}(j)}{n^j} \right|=1-C_2 \quad \textrm{and} \quad \limsup_{j \to +\infty} \left| \frac{\calL_{\textbf{a}'}(j)-\calL_{\textbf{b}'}(j)}{n^j} \right|=1-C_1.
$$

\medskip

\textit{Case C: $e(j)-d(j)=-1$.} Again, by \eqref{eq335}, we have
$$
 \frac{\calL_{\textbf{a}'}(j)-\calL_{\textbf{b}'}(j)}{n^j} =  \frac{\calL_{\textbf{a}}(j)-\calL_{\textbf{b}}(j)}{n^j}+\frac{N_2-N_1}{n^j}-1, \quad \textrm{for $j$ large enough}.
$$
Similarly, we have
$$
\frac{\calL_{\textbf{a}}(j)-\calL_{\textbf{b}}(j)}{n^j}>0,  \quad \textrm{for $j$ large enough},
$$
and hence
$$
\left|  \frac{\calL_{\textbf{a}'}(j)-\calL_{\textbf{b}'}(j)}{n^j} \right|=1- \frac{\calL_{\textbf{a}}(j)-\calL_{\textbf{b}}(j)}{n^j}-\frac{N_2-N_1}{n^j},  \quad \textrm{for $j$ large enough}.
$$
Thus, for the third case, we still have
$$
\liminf_{j \to +\infty} \left| \frac{\calL_{\textbf{a}'}(j)-\calL_{\textbf{b}'}(j)}{n^j} \right|=1-C_2 \quad \textrm{and} \quad \limsup_{j \to +\infty} \left| \frac{\calL_{\textbf{a}'}(j)-\calL_{\textbf{b}'}(j)}{n^j} \right|=1-C_1.
$$

Finally, we turn back to the general case when $N_1, N_2 \in \Z$. However, this follows in the same manner, as when comparing two representations, we can always rename them so that we have $\delta_1' = \delta+N_1$, with $N_1\geq 0$ (and similarly for $\delta_2$).  Then with the renamed constants $C_1'$ and $C_2'$, we can run through the same argument and get the same conclusion.

\end{proof}


\begin{thebibliography}{99}

\bibitem{A}  T.C. Anderson, A framework for Calder\'on-Zygmund Operators on Spaces of Homogeneous Type.  Ph.D. thesis, Brown University, 2015.

\bibitem{AV} T.C. Anderson and A. Vagharshakyan.  A simple proof of the sharp weighted estimate for Calder\'n-Zygmund operators on homogeneous spaces. \emph{J. Geom. Anal}. 24 (2014), no. 3, 1276--1297.

\bibitem{DCU} D. Cruz-Uribe, OFS.  Two weight norm inequalities for fractional integral operators and commutators.  Preprint available on arXiv.

\bibitem{HK}  T. Hyt\"nen and A. Kairema. Systems of dyadic cubes in a doubling metric space. \emph{Colloq. Math}. 126 (2012), no. 1, 1--33.

\bibitem{AL} A.K. Lerner, A simple proof of the $A_2$ conjecture, \emph{Int. Math. Res. Not.}, 2013, no. {\bf 14,} 3159--3170. 

\bibitem{LN}  A.K. Lerner and F. Nazarov.  Intuitive dyadic calculus: the basics.  \emph{Expo. Math}. 37 (2019), no. 3, 225--265.

\bibitem{LPW} J. Li, J. Pipher, L.A. Ward, Dyadic structure theorems for multiparameter function spaces, \emph{Rev. Mat. Iberoam}. {\bf 31} (2015), no. 3, 767--797.

\bibitem{P} C. Pereyra, Weighted inequalities and dyadic harmonic analysis. 
\emph{Excursions in harmonic analysis. Volume 2}, 281--306, Appl. Numer. Harmon. Anal., Birkhauser/Springer, New York, 2013.

\bibitem{PW} J. Pipher and L. Ward,  $BMO$ from dyadic $BMO$ on the bidisc.  \emph{Journal London Math. Soc}., Vol. 77 No.2, 2008, p. 524--544.

\bibitem{TM} T. Mei, $BMO$ is the intersection of two translates of dyadic $BMO$.  \emph{C.R. Acad.
Sci. Paris, Ser}. I {\bf 336} (2003), 1003--1006.

\end{thebibliography}
\end{document}